\documentclass[12pt]{article}

\usepackage{amsmath}
\usepackage{amssymb}
\usepackage{amsfonts}

\usepackage{amsthm}

\newtheorem{theo}{Theorem}

\newtheorem{lemm}{Lemma}

\newcommand{\Po}{{\cal P}}

\renewcommand{\xi}{Z}
\renewcommand{\alpha}{r}

\newcommand{\N}{\mathbb{N}}

\newcommand{\R}{\mathbb{R}}
\newcommand{\E}{\mathbb{E}}

\renewcommand{\Pr}{\mathbb{P}}

\renewcommand{\E}{\mathbb E \,}

\newcommand{\eqco}{\setcounter{equation}{0}}

\newcommand{\F}{{\cal F}}

\newcommand{\txi}{{\tilde{\xi}}}

\newcommand{\eps}{\varepsilon}
\newcommand{\bdm}{\begin{displaymath}}
\newcommand{\edm}{\end{displaymath}}
\newcommand{\benu}{\begin{enumerate}}
\newcommand{\eenu}{\end{enumerate}}
\newcommand{\beqn}{\begin{equation}}
\newcommand{\eeqn}{\end{equation}}
\newcommand{\be}{\begin{equation}}
\newcommand{\ee}{\end{equation}}
\newcommand{\bea}{\begin{eqnarray}}
\newcommand{\eea}{\end{eqnarray}}
\newcommand{\bean}{\begin{eqnarray*}}
\newcommand{\eean}{\end{eqnarray*}}
\newcommand{\bear}{\begin{eqnarray}}
\newcommand{\eear}{\end{eqnarray}}

\renewcommand{\epsilon}{\varepsilon}
\renewcommand{\R}{\mathbb{R}}

\renewcommand{\P}{{\mathbb P}}

\newcommand{\rhodist}{\mu}

\newcommand{\la}{{\lambda}}

\begin{document}
\title{\bf Non-triviality of the vacancy phase transition for the Boolean model}

\author{
Mathew D. Penrose$^{1}$ \\
{\normalsize{\em University of Bath}} }

 \footnotetext{ $~^1$ Department of
Mathematical Sciences, University of Bath, Bath BA2 7AY, United
Kingdom: {\texttt m.d.penrose@bath.ac.uk} }

%Running title: {\bf } \\


%\footnotetext{ Keywords: semi-homogeneous random digraph,
% giant component,branching process}

%\date{}
\maketitle

%\newpage
\begin{abstract}
In the spherical Poisson Boolean model, one takes the union of
random balls centred on the points of a Poisson process in 
Euclidean $d$-space  with $d \geq 2$.
We prove that whenever the radius distribution
has a finite $d$-th moment, there exists 
a strictly positive value for the intensity such that
the vacant region percolates.
\end{abstract}

%\section{Statement of results}
%\label{StateResult}

\section{Introduction}
\label{SecIntro}
The Boolean model \cite{Gilbert,Hall} 
is a classic model of continuum percolation
\cite{MR,BR}  
and more general
stochastic geometry \cite{Hallbk,CSKM,SW,LP}. 
In the spherical version of this model,  an {\em occupied
region} in Euclidean $d$-space 
is defined as a union of balls (sometimes called {\em grains})
 of fixed or random radius
 centred on the points of a Poisson process of intensity $\lambda$.
One may  define a
critical value $\lambda_c$ of $\lambda$,
 depending on the radius distribution,
 above which the occupied region percolates, and a further
critical value $\lambda_c^*$, below which 
 the complementary {\em vacant region} percolates.
It is a fundamental question   whether these critical values
 are {\em non-trivial}, i.e. strictly positive and finite. 

For fixed or bounded radii,  the non-triviality of $\lambda_c$
and $\lambda_c^*$ for $d \geq 2$ is well known and may be proved
using discretization and  counting arguments from
lattice percolation  theory. For unbounded radii, it took some years
to fully characterize those radius distributions for which
$\lambda_c$ is non-trivial \cite{Hall,Gouere}. 
In the present work we carry out a similar task for $\lambda_c^*$.

We now describe the model in more detail
(for yet more details we refer the reader to \cite{MR} or \cite{LP}).
Let $d \in \N$ with $d \geq 2$. Let $\rhodist$ be a probability
measure on $[0,\infty)$
with $\rhodist(\{0\}) <1$.
Let $\lambda \in (0,\infty)$.
On a suitable probability space $(\Omega,\F,\P)$
(with associated expectation operator $\E$),
let $\Po_\lambda= \{y_k: k \in \N\}$ be a homogeneous
 Poisson point process in $\R^d$ of
intensity $\lambda$ (here viewed as a random subset of
$\R^d$ enumerated in order of increasing distance from the origin),
 and let $\rho,\rho_1,\rho_2,\ldots$ be independent nonnegative
random variables with common distribution $\rhodist$,
 independent of $\Po_\lambda$.  For $x \in \R^d$ and $r \geq 0$ we let
 $B(x,r):= \{y \in \R^d:\|y-x\| \leq r\}$,
where $\|\cdot\|$ is the Euclidean norm.
The occupied and vacant regions of the (Poisson, spherical) 
 Boolean model are  random sets $\xi_\lambda \subset \R^d$ and
$\xi^*_\lambda \subset \R^d$, given respectively by 
$$
\xi_\lambda = \cup_{y_k \in \Po_\lambda} B(y_k,\rho_k);
~~~~~~
\xi^*_\lambda = \R^d \setminus \xi_\lambda.
$$
Let $U_\lambda$ be the event that $\xi_\lambda$ {\em percolates}, i.e.
has an unbounded
connected component, and
let $U^*_\lambda$ be the event that $\xi^*_\lambda$ percolates.
%has an unbounded connected component.
 By an ergodicity argument (see \cite{MR}, or \cite{LP}, Exercise 10.1), 
$\Pr[U_\lambda] \in \{0,1\}$
and $\Pr[U^*_\lambda] \in \{0,1\}$. Also
$\Pr[U_\lambda] $ is increasing in $\lambda$,
 while $\Pr[U^*_\lambda]$ is
decreasing in $\lambda$.
Define the critical values
$$
\lambda_c:= \inf\{\lambda:\Pr[U_\lambda]= 1\};
~~~~~
\lambda_c^*:= \inf\{\lambda:\Pr[U^*_\lambda]= 0\}.
$$
%Provided $\Pr[\rho=0]<1$, 
It is well known that
 $\lambda_c$ and $\lambda_c^*$ are finite, and that
if $\E[\rho^d] = \infty $ then $\xi_\lambda = \R^d$ almost surely, for
any $\lambda >0$ (see \cite{Hall}, \cite{MR} or \cite{LP}), so that
$\lambda_c = \lambda_c^* =0$.
 Hence $\E[\rho^d] < \infty$ is a necessary 
condition for $\lambda_c$ or $\lambda_c^*$ to be
strictly positive.
In the case of $\lambda_c$,
 Gou\'er\'e \cite{Gouere} has shown that this condition
is also sufficient:
\begin{theo} 
\label{thm:Gouere}
{\rm \cite{Gouere}}
If $\E[\rho^d] < \infty$ then $\lambda_c >0$.
\end{theo}
We here present a similar result for $\lambda_c^*$:
\begin{theo} 
\label{thm:main}
If $\E[\rho^d ] < \infty$ then $\lambda^*_c >0$.
\end{theo}
Theorem \ref{thm:main} says that
 for the spherical Poisson Boolean model
with
$\E[\rho^d ] < \infty$,
there exists a non-zero value of the intensity $\lambda$ for
which the vacant region percolates.   
In fact we can say more:
\begin{theo}
\label{thm:sharp}
For any $\mu$, if $d=2$ then $\lambda^*_c = \lambda_c$.
If $d\geq 3$ then $\lambda_c^* \geq \lambda_c$.
\end{theo} 
Sarkar \cite{Sarkar} has proved the strict inequality $\lambda^*_c > \lambda_c$
for $d \geq 3$ when $\rho$ is deterministic, i.e. when $\mu$ is
a Dirac measure.

% For $x \in \R$ we set  $\log^+(x):= \log(\max(x),1)$.
Theorem \ref{thm:main} could be seen as a trivial corollary
of Theorems \ref{thm:Gouere} and \ref{thm:sharp}.
However, we would like to prove Theorems
\ref{thm:main} and \ref{thm:sharp} separately, to emphasise
that our proof of Theorem \ref{thm:main} is  
self-contained
 (and quite short), 
whereas our proof of Theorem
\ref{thm:sharp} is not, as we now discuss.

In parallel and independent work,
 Ahlberg, Tassion and Teixeira \cite{ATT2}  prove
a similar set of results to our Theorems \ref{thm:main}
and \ref{thm:sharp}; their proof seems to be completely
different from ours. Earlier,
 in \cite{ATT} they proved
for $d=2$ that (among other things) $\lambda_c^* = \lambda_c$ 
whenever $\E[\rho^2 \log \rho] <\infty $.

We prove Theorem \ref{thm:main} in the next two sections.
 The proof of Theorem \ref{thm:sharp} 
%(for $d=2$) 
is given by adapting
 our proof of Theorem \ref{thm:main} using results in 
\cite{ATT}, and is therefore heavily reliant on \cite{ATT};
we give this argument
in Section \ref{secsharp}.

Finally, we consider the relation between
 $\lambda_c^*$
and a different percolation threshold, defined in terms
of expected diameter. 
For non-empty $B \subset \R^d$, let $D(B):= \sup_{x,y \in B} (\|x-y\|)$,
the Euclidean diameter of $B$, and  set $D(\emptyset)=0$.
Let $W_\la$ be the connected component of $\xi_\lambda $ containing the
origin, and set
$$
\lambda_D := \inf \{\lambda: \E[D(W_\lambda)] = \infty\}.
$$ 
It is easy to see that that $\lambda_D  \leq \lambda_c$. 
Therefore by Theorem \ref{thm:sharp}, 
for any $\mu$ we have
\bea
\label{eq:D}
\lambda_c^* \geq \lambda_D.
\eea
In Section \ref{secfirstpf} we present an alternative, 
 rather simple, direct proof of (\ref{eq:D}) (not
reliant on any other results, either here or in
\cite{ATT}).

A further result in \cite{Gouere} says that
$\lambda_D >0$, if and only if 
$\E[\rho^{d+1}] < \infty$. Therefore 
(\ref{eq:D}) provides an alternative proof 
that $\lambda_c^*>0$ under this stronger moment
condition.  Moreover, it is known in many cases 
 that $\lambda_D = \la_c$ (see e.g. \cite{MR,Z,DCRT}),
and in all such cases our proof
of (\ref{eq:D}) provides another way to
show that $\lambda_c^* \geq \lambda_c$.

Our proof of Theorems \ref{thm:main} and \ref{thm:sharp}
for $d=2$
uses a form of multiscale methodology,
 inspired by \cite{Gouere},
which may be of use in other settings. 
We conclude this section with an outline of the method.
 At length-scale
$\alpha$, we define functions $f(\alpha)$ and $g(\alpha)$.
Up to a constant multiple, $f(\alpha)$ is the probability of
a `local' event (defined in terms of a box-crossing, using
only grains centred near the box)  
while $g(\alpha)$ is
the probability of an `outside influence' event
that is still determined at length-scale $\alpha$. 

We show that  $g(10^n)$ is summable in $n$ (see Lemma \ref{lemF} below),
 and also that $f(10^{n+1}) \leq f(10^n)^2 + g(10^{n+1})$ (see (\ref{0530a}) 
and (\ref{0530d}) below). From this 
we can deduce that there exists $n_0$ such that
$\sum_{n \geq n_0}  (f(10^n)+g(10^n)) < 1$, if only we can get
started by showing $f(n_0)$ is sufficiently small. 
This can be done either by taking $\lambda$ small (in the
proof of Theorem \ref{thm:main}) or for general $\lambda < \lambda_c$,
 by taking $n_0$ large  and using a result from \cite{ATT}
(in the proof of Theorem \ref{thm:sharp}). Finally, 
we can take a sequence of boxes of length $10^{n+n_0}$, such
that if none of these is crossed then $\xi_\la^*$ percolates.

We let $o$ denote the origin in $\R^d$,
and for $r >0$ put $B(r) := B(o,r)$.

\section{Preparation for the proof}
\eqco
  
Throughout this section
we assume that $d=2$.
 We give some definitions and lemmas required for our  proof
of Theorem \ref{thm:main}.

Given $\lambda >0$, for each Borel set $A \subset \R^2$ we define the 
random set 
$$
\xi^A_\lambda := \cup_{\{k:y_k \in \Po_\lambda \cap A\}} B(y_k,\rho_k).
$$
Also, for $r >0$ set
$
A_r:=  \cup_{x \in A} B(x,r),
$
the (deterministic) $r$-neighbourhood of $A$.

Given $\alpha >0$, let 
$S(\alpha):= [-5\alpha,5\alpha]\times [-\alpha/2,\alpha/2]$, 
 the closed $10 \alpha \times \alpha$ horizontal 
rectangle (or `strip') centred at $o$. 
Note that $S(\alpha)_\alpha$ is a $12 \alpha \times 3 \alpha$
rectangle with its corners smoothed (this smoothing is not important to us).

Let $F_{\lambda}(\alpha)$
be the event that there is a short-way crossing of 
$S(\alpha)$ by $\xi_{\lambda}^{S(\alpha)_\alpha}$ (that is,
by grains centred within the $\alpha$-neighbourhood of $S(\alpha)$).
Also define the event 
\bea
G_{\lambda}(\alpha) = \{ 
 \xi_\lambda^{B(10^6\alpha) \setminus S(\alpha)_\alpha}
 \cap S(\alpha)  \neq \emptyset \}.
\label{Gdef}
%\\
%G'_{\lambda}(x,\alpha) :=  \{
% \xi_\lambda^{B(x,10^6 \alpha) \setminus R'(x,5\alpha)}
% \cap R'(x,\alpha)  \neq \emptyset \}.
%\nonumber
\eea
%
%We use the notation $F_{\lambda}(\alpha):= F_{\lambda}(o,\alpha)$ 
%and $G_{\lambda}(\alpha) := G_{\lambda}(o,\alpha)$. 

\begin{lemm}
\label{LemC1}
 There is a constant $C_1 \geq 1$ such that for all $\lambda >0$ and
 $\alpha  >0  $,
\bea
\Pr[ F_{\lambda}(10 \alpha) ] \leq C_1 ( 
\Pr[ F_{\lambda}( \alpha) ]^2 + 
\Pr[ G_{\lambda}( \alpha) ]).
\label{0530a}
\eea
\end{lemm}
\begin{proof}
Fix $(\lambda,\alpha)$.
Set $S := S(10\alpha) = [-50 \alpha ,50 \alpha] \times [-5\alpha ,5 \alpha]$.
Let $T:= [-50\alpha,50\alpha] \times [-4.5 \alpha,-3.5\alpha]$
and 
 $\tilde{T}:= [-50\alpha,50\alpha] \times [3.5 \alpha,4.5\alpha]$,
so that $T$ and $\tilde{T}$ are horizontal $100 \alpha \times \alpha$
thin strips along $S$ near the bottom and top of $S$, respectively.

We shall now define a collection $R_1,\ldots,R_{37}$ of horizontal $10\alpha
\times \alpha$ and
vertical  $\alpha \times 10\alpha$  rectangles
 that knit together in such a way that
if there is a long-way vacant crossing of each of $R_1,\ldots,R_{37}$
then there is a long-way vacant crossing of $T$ (this
is a well known technique in these kinds of proof). 
We shall arrange that they are all contained
within the band $\R \times [-12 \alpha,-2\alpha]$ and their
$\alpha$-neighbourhoods $(R_1)_\alpha,\ldots, (R_{37})_\alpha)$
all lie within the lower half of the region 
$S_{10\alpha} := (S(10\alpha))_{10 \alpha}$.

Here are the details.
Let $R_1,R_2,\ldots,R_{19}$ be horizontal $10 \alpha \times \alpha$
rectangles centred on  $(-45 \alpha,-4 \alpha),(-40\alpha, -4\alpha),
\ldots,(45 \alpha,-4\alpha)$
respectively. Let $R_{20},\ldots,R_{37}$ be vertical $\alpha \times 10 \alpha$
rectangles centred at 
 $(-42.5\alpha,-7 \alpha), (-37.5\alpha,-7\alpha),
\ldots, (42.5,-7\alpha) $ 
respectively.

Similarly, we define a collection $R_{38},\ldots,R_{74}$
of $10 \alpha \times \alpha $ and $\alpha \times 10  \alpha$
rectangles, such that if each of these has a long-way vacant
crossing then there is a long-way vacant crossing of $\tilde{T}$.
Each rectangle $R_{37+i}$, $1 \leq i \leq 37$,  is defined 
simply as the reflection  of $R_i$ in the $x$-axis. 

For $1 \leq i \leq 74$,
 let $D_i$  be the disk of radius $10^6\alpha$
 with the same centre as $R_i$.
Let $A_i$  denote the event that  there exists a grain of the Boolean model
that intersects $R_i$
and has its centre in the region
 $D_i   \setminus (R_i)_{\alpha}$.
 Let $B_i $ denote the event that the rectangle $R_i$ can be crossed the short
way in the union of grains
that are centred inside $(R_i)_\alpha$.
If $R_i$ is crossed the short way in the union of grains centred
in $D_i$, then $A_i \cup B_i$ must occur.

Suppose $F_\la(10r)$ occurs, i.e. there is a
 short-way occupied crossing of $S$,
using grains centred in $S_{10\alpha}$.
Then there is no long-way vacant crossing of $S$, and hence 
 no long-way vacant crossing either of $T$ or of $\tilde{T}$.
Hence
\bea
F_\la(10r) & \subset & \cup_{(i,j) \in \{1,\ldots,37\}^2}
\left(
(A_i \cup B_i)\cap(A_{37+j} \cup B_{37+j}) \right)
\nonumber \\
 & \subset &
\left(
\cup_{i=1}^{74}A_i 
\right) \cup \left(
\cup_{(i,j) \in \{1,\ldots,37\}^2} 
(B_i \cap B_{37+j})\right).
\label{0606a}
\eea
%Hence for some $(i,j) \in \{1,\ldots,37\}^2$, events
%$A_i \cup B_i$ and $A_{37+i} \cup A_{37+j}$ occur.
%Hence either at least one of the events $A_1,\ldots,A_{74}$ occurs,
% or there is a pair $(i,j) \in \{1,\ldots,37\}^2$
%such that events $B_i$ and $B_{37+j}$ both occur. 
%
For $i,j \in \{1,\ldots,37\}$,
 since $(R_i)_\alpha \cap (R_j)_\alpha = \emptyset$
the events $B_i$ and $B_{37+j}$ are independent.
Hence by (\ref{0606a}) and the union bound
we have (\ref{0530a}), taking $C_1 = 37^2$. 
\end{proof}

\begin{lemm}
\label{lemF}
Suppose $\E[\rho^2]< \infty$. Let $\lambda_0 \in (0,\infty)$.
Then
\bea
\sum_{n \geq 1} \sup_{\lambda \in (0,\lambda_0] }
 \Pr[G_{\lambda} (10^n)] < \infty.
\label{0530b}
\eea
\end{lemm}
\begin{proof}
Given $\lambda,\alpha >0$,
 if $G_\lambda(\alpha)$ occurs 
then there exists a point $y_k \in \Po_\lambda  \cap B(10^6 \alpha) \setminus
S( \alpha)_\alpha$ with associated radius $\rho_k > \alpha$.
Therefore by Markov's inequality $\Pr[G_\lambda(\alpha)]$
is bounded above by the expected number of such points $y_k$.
Therefore
$$
\Pr[G_\la(r)] \leq \la \pi (10^6\alpha)^2 \Pr[\rho > \alpha]
= 10^{12} \la \pi 
\alpha^2 \Pr[\rho^2 > \alpha^2].
$$   
Hence,
\bean
\sum_{n \geq 1} \sup_{\lambda \in (0,\lambda_0] }
 \Pr[G_{\lambda} (10^n)] 
\leq 10^{12} \lambda_0 \pi
 \sum_{n=1}^\infty 100^n \Pr[\rho^2 > 100^n]
\eean
which is finite because we  assume $ \E[\rho^2]< \infty$.
\end{proof}
\begin{lemm}
\label{lem0601}
Suppose $\E[\rho^2]< \infty$. 
 Then there exist $b >0$ and $\lambda >0$ such that
\bea
\sum_{n=1}^\infty (\Pr[F_\lambda(10^nb) ] + \Pr[G_\lambda(10^nb) ])
\leq 1/2.
\label{0601d}
\eea 
\end{lemm}
\begin{proof}
Let $C_1 \geq 1$ be as in Lemma \ref{LemC1}.
Given $\lambda,\alpha >0$   we define
$$
f_{\lambda}(\alpha) := C_1 \Pr[F_{\lambda}(\alpha)];~~~~~
g_{\lambda}(\alpha) := C_1^2 \Pr[G_{\lambda}(\alpha/10)].
$$
Then by (\ref{0530a}) we have
\bea
f_{\lambda}(\alpha) \leq  C_1^2 ( \Pr[F_{\lambda}( \alpha/10)]^2
+ \Pr[G_\la( \alpha/10)] )
\nonumber \\
= f_{\lambda}(\alpha/10)^2 + g_{\lambda}(\alpha) .
\label{0530d}
\eea
Let $C_2=9$.  Using
(\ref{0530b}), we can choose $b$ to be a big enough power of 10
so that 
for all $ \lambda \in (0 ,1]$, we have
% firstly,  
% $g_{\lambda}(\alpha) \leq 1/(2C_2)$ for all $ \alpha \geq b$, and
%secondly,
\bea
\sum_{n=1}^\infty g_{\lambda}(10^nb) \leq C_2^{-2}.
\label{0530f}
\eea
Now fix this $b$.  Choose $\lambda \leq 1$ to be small enough
 so that $f_{\lambda}(b) \leq C_2^{-1}$.  
Using (\ref{0530d}) repeatedly,
 we have $f_{\lambda}(10^n b) \leq C_2^{-1} $ for all  $n$.
Then using (\ref{0530d}) repeatedly again,
% as in \cite[equation (11)]{Gouere} 
we have for $n \in \N$ that
\bean
f_{\lambda}(10^nb) 
& \leq & \frac{f_\lambda(10^{n-1}b)}{C_2} + g_\lambda(10^{n} b)  \leq \cdots 
%\nonumber 
\\
& \leq & C_2^{-n-1} + \frac{g_{\lambda}(10b)}{C_2^{n-1}} + \frac{g_{\lambda}(100b)}{C_2^{n-2}}
+ \cdots + \frac{g_{\lambda}(10^nb)}{C_2^0},
\eean
and therefore
\bean
\sum_{n \geq 1} f_{\lambda}(10^nb) \leq (C_2^{-2} + g_{\lambda}(10b)+
 g_{\lambda}(100b) + g_{\lambda}(1000b)+\cdots)
\\
\times(1 + C_2^{-1}+ C_2^{-2} + \cdots)
%\\
%\leq 2 C_2^{-1} +2 \sum_{n=1}^\infty g_{\lambda}(10^nb).
\eean 
so by 
  (\ref{0530f}) and the fact that $\sum_{k=0}^\infty C_2^{-k} \leq 2$,
we have
\bean
\sum_{n \geq 1} ( \Pr[ F_{\lambda}(10^nb) ] + \Pr[G_{\lambda}(10^nb) ] )
\leq \sum_{n\geq 1} (f_{\lambda}(10^nb) + g_{\lambda}(10^nb))
\\
\leq 2C_2^{-2} + 3 \sum_{n \geq 1} g_\lambda(10^n b)
 \leq 5 C_2^{-2},
%\leq 2C_2^{-2} + 3C_2^{-1}
\eean
and hence (\ref{0601d}).
%as required.
\end{proof}

\section{Proof of Theorem \ref{thm:main}}
\label{secmainpf}
%\eqco
 We can now complete the proof of Theorem \ref{thm:main}. 
We assume from now on that
 \bea
%\E[\rho^d \log^+ (\rho) ] < \infty,
\E[\rho^d ] < \infty.
\label{dmom}
\eea

Consider first the case with $d=2$. 
Let $b$ and $\lambda$ be as given in Lemma \ref{lem0601}.

Let $S_1,S_2,S_3,\ldots$
be a sequence of `strips', i.e. closed  rectangles of aspect
ratio 10, with successive lengths (the short way) $ 10b, 100b, 1000b, \ldots$
alternating between horizontal and vertical strips with
each strip $S_n$ centred at the origin. Then  each
strip
 $S_n$ crosses the next one $S_{n+1}$ the short way.

For each $n \in \N$, define the events
$$
H_n: = \{ S_n  {\rm~ is~   crossed~ by~ }
\xi_\lambda^{S_{n+2}}
 {\rm ~the~ short ~way} \};
$$
$$
J_n := \{ \xi_{\lambda}^{S_{n+4} \setminus S_{n+2} }  \cap S_n 
\neq \emptyset \}.
$$
\begin{lemm}
\label{lemifnone}
 If none of the events $H_1,J_1,H_2,J_2,\ldots$ occurs then
$\xi^*_\lambda$ percolates.
\end{lemm}
\begin{proof}
 Suppose none of the events $H_1,J_1,H_2,J_2,\ldots$ occurs.

We claim for each $n \in \N$ that 
$\xi^{\R^2 \setminus S_{n+2}}_{\lambda} \cap S_n = \emptyset$. Indeed,
if  
$\xi^{\R^2 \setminus S_{n+2}}_{\lambda} \cap S_n \neq \emptyset$ then for
some integer  $m \geq n +2$ with $m-n$ even
 we have 
$\xi^{S_{m+2}  \setminus S_{m}}_{\lambda} \cap S_n \neq \emptyset$,
and then since $n \leq m$ we also have 
$S_n \subset S_m$ so that
$\xi^{S_{m+2}  \setminus S_{m}}_{\lambda} \cap S_m \neq \emptyset$,
contradicting the assumed non-occurrence of $J_{m-2}$.

For each $n$,
by the assumed non-occurrence of $H_n$ along with the preceding claim
there   is no short-way crossing of $S_n$ by $\xi_\lambda$ so
there is a long-way crossing of $S_n$ by $\xi_{\lambda}^*$, i.e.
a path $\gamma_n \subset S_n \cap \xi_\lambda^*$ that crosses
$S_n$ the long way. 

Then for
each $n$ we have $\gamma_n \cap \gamma_{n+1} \neq \emptyset$,
so $\cup_n \gamma_n$ is an unbounded connected set contained in
$\xi_{\lambda}^*$. Therefore $\xi_\lambda^*$ percolates.  
\end{proof}

%\begin{proof}[Proof of Theorem \ref{thm:main} for $d=2$]
\begin{proof}[Proof of Theorem \ref{thm:main}]

Suppose $d=2$.
Let $\lambda$ and $b$ be as given in Lemma \ref{lem0601}.
Recall  the definition of events $F_\lambda(\alpha)$ and
$G_\lambda(\alpha)$ at (\ref{Gdef}). 
We claim now for each $n$ that  
\bea
\Pr[H_n \cup J_n] \leq \Pr[ F_\la(10^nb)] + \Pr[G_\la(10^nb)].
\label{0605a}
\eea
Indeed, suppose the parity of $n$ is such that $S_n$ is
horizontal. Then, in terms of earlier notation, $S_n= S(10^nb)$.
Since $S_{n+4} \subset B(10^{6+n}b)$
we have $J_n \subset G_\la(10^n b)$ and
 $H_n \subset F_\la(10^n b)  \cup G_\la(10^n b)$.
Then (\ref{0605a}) follows from the union bound.

 Using first Lemma \ref{lemifnone}, then
(\ref{0605a}), 
and finally (\ref{0601d}),
we have
\bean
1 - \Pr[U^*_\lambda ]
 & \leq & \Pr[\cup_{n=1}^\infty (H_n \cup J_n)]  
\\
& \leq  & \sum_{n=1}^\infty 
(\Pr[ F_{\lambda} (10^n b) ] 
+ \Pr[ G_{\lambda} (10^n b) ] ) \leq 1/2. 
\eean
Therefore by ergodicity
 $\Pr[U^*_\lambda]=1$ so $\lambda \leq \lambda_c^*$. Hence
we have $\lambda_c^* >0$ as required.
%\end{proof}

%\begin{proof}[Proof for $d \geq 3$]
%
Now suppose $d \geq 3$. Let $\txi_\la$, be  the intersection
of $\xi_\la$  with the two-dimensional
subspace $\R^2 \times \{o''\}$ of $\R^d$, 
where $o''$ denotes the origin in $\R^{d-2}$.

Let $\omega_{d-2}$ denote the volume of the unit ball
in $\R^{d-2}$. It can be seen that
$\txi_\lambda$
%of $\xi_\lambda$ to the $(x,y)$-plane
 is a two-dimensional
Boolean model with intensity 
\bean
\lambda \omega_{d-2} (d-2)
 \int_0^\infty \Pr[\rho \geq r] r^{d-3}dr
%\\
=
\lambda \omega_{d-2} 
\E[\rho^{d-2} ] =: \lambda',
\eean
which is finite by our assumption (\ref{dmom}).
Moreover if $\sigma$ denotes a random variable with the
radius distribution in this planar Boolean model we claim that
$\E[\sigma^2] < \infty$.   This
can be demonstrated by a computation, but it is more quickly
seen using the fact that, since  
$\Pr[o \in \xi_\lambda]<1$ for the original
Boolean model by 
%our original assumption 
(\ref{dmom}), 
 also $\Pr[o \in \txi_\lambda] <1$, which would
not be the case if $\E[\sigma^2]$ were infinite.

 Therefore by the
two-dimensional case already considered, for small enough $\lambda >0$
we have $\lambda'$ small enough so that
 the complement
(in the space $\R^2 \times \{o''\}$)
 %$(x,y)$-plane)
of %the restriction of
 $\txi_\lambda$
percolates. Hence $\xi_\lambda^*$ percolates for
small enough $\lambda >0$, so $\lambda^*_c > 0$.
\end{proof}

\section{Proof of Theorem \ref{thm:sharp}}
\label{secsharp}
As mentioned in Section \ref{SecIntro}, if $\E[\rho^d] =\infty$ then
$\lambda_c = \lambda_c^*=0$, so 
without loss of generality we assume (\ref{dmom}).
%(else $\lambda_c^* = \lambda_c =0$).
%whenever $\E[\rho^d]< \infty$.

First suppose  $d=2$. We need to prove that $\lambda_c^* = \lambda_c$.

Suppose $\lambda > \lambda_c$. 
Let $V_\lambda^*$ be the event
that there is an unbounded component of $\xi_\la^*$ intersecting with
 $B(1)$.
For $n \in \N$, set $Q(n) := [-n,n]^2$.  Let
 $E(n)$ be the event that there exists a path in 
$\xi_\la^* $ from $Q(n)$ to
%$B(n)$ to
 $\R^2 \setminus Q(3 n )$.

The annulus $Q(3n) \setminus Q(n)$ can be written as the union
of two $3n \times n$ and two $n \times 3n$
 rectangles, and if $\xi_\la$ crosses each of these four rectangles
 the long way then $Q(n)$ is  surrounded
by an  occupied circuit contained in $Q(3n)$ so  $E(n)$ does not occur.  
Hence by Theorem 1.1 (i) of \cite{ATT} and the union bound, $\Pr[E(n)] \to 0$
as $n \to \infty$.  Since 
$
V^*_\lambda \subset \cap_{n =1}^\infty E(n),
$
we therefore have $\Pr[V^*_\la] =0$ and hence $\Pr[U_\la^*]=0$. 
Hence $\lambda \geq \lambda_c^*$ so $\lambda_c^* \leq \lambda_c$. 

Now suppose $\lambda < \lambda_c$. Then by Theorem 1.1(iii) of 
\cite{ATT}, in the proof of our Lemma \ref{lem0601} we can choose 
$b$ large enough 
so that we have both (\ref{0530f}), and the inequality
 $f_\lambda(b) < C_2^{-1}$.
Then the rest of the proof of Lemma \ref{lem0601} carries through
for this $(b,\lambda)$, 
 so the conclusion of
  Lemma \ref{lem0601}  holds for this $(b,\lambda)$.
 Then the proof
(for $d=2$)
 in Section \ref{secmainpf} 
 works for this $(b,\lambda)$, showing that
 $\lambda \leq \lambda^*_c$
 for any $\lambda < \lambda_c$ 
and hence that 
$\lambda_c^* \geq \lambda_c$.
Thus $\lambda_c^* = \lambda_c$ for $d=2$.

Now suppose that $d \geq 3$ and $\lambda < \lambda_c$.
 Then as discussed in Section \ref{secmainpf},
 $Z_\la \cap (\R^{d-2} \times \{o''\})$  
is a two-dimensional Boolean model possessing no infinite
component, so  the radius distribution for this
two-dimensional Boolean model has finite second moment and
  the intensity $\lambda'$ of this two-dimensional Boolean
model is subcritical
(in fact, strictly subcritical since we can repeat the argument
for any $\lambda_1 \in (\lambda,\lambda_c)$).
 Therefore by the argument just given for
$d=2$,
the complement (in $\R^{d-2} \times \{o''\}$) of this
Boolean model percolates, and  therefore 
the original $\xi^*_\lambda$ also  percolates so $\lambda \leq \lambda_c^*$.
Hence $\lambda_c^* \geq \lambda_c$, and the proof is complete.

\section{Alternative proof of (\ref{eq:D})}
% Theorem \ref{thmD} }
\label{secfirstpf}
%\eqco

%In this section we give this argument.

We divide the nonnegative $x$-axis into unit intervals $I_0,I_1,I_2,\ldots$
where $I_k = [k,k+1) \times \{o'\}$ (here $o'$ is
the origin in $\R^{d-1}$).
For each $k \in \N$ let $W_{k,\lambda}$ be the union of $I_k$ and
all components of $\xi_\lambda$ which intersect $I_k$.
%
%For each $k \in \N$ let
%Let $W_k$ be the union of $B(1/2)$ and all components of
%$\xi_\lambda $ which intersect $B(1/2)$.  
%
%For non-empty $B \subset \R^d$, let $D(B):= \sup_{x,y \in B} (\|x-y\|)$,
%the Euclidean diameter of $B$.
\begin{lemm}
\label{lemGouere}
If $0< \lambda < \lambda_D$,
then $\E[D(W_{0,\lambda})] < \infty$.
\end{lemm}
\begin{proof}
Fix $\lambda \in (0, \lambda_D)$. Then $\E[D(W_\lambda)] <\infty$.
Let $F$ be the event that $I_0 \subset \xi_\lambda$,
and set $F^c:= \Omega \setminus F$. Then $0 <\Pr[F] <1$.
If $F$ occurs then $W_{0,\lambda} = W_\lambda$. Hence
by the Harris-FKG inequality (see \cite{MR} or \cite{LP}),
$$
%\E[D(W_{0,\lambda})|F^c]
%\leq
\E[D(W_{0,\lambda})]
\leq
\E[D(W_{0,\lambda})|F] = \E[D(W_\la)|F] < \infty, 
$$
as required.
%and hence by the law of total probability, we have
% $\E[D(W_{0,\lambda})] < \infty$. 
\end{proof}

%Assuming (\ref{strongmoment}), there exists $\lambda_0>0$ such that
%for all $\lambda \leq \lambda_0$  we have
%\bea
%\E [D(S)] < \infty.
%\eea

Given $\lambda >0$, define the event
$$
E_\lambda:= \left( \cap_{k=2}^\infty \{ D(W_{k,\lambda}) \leq k/2 \}
\right) \cap
\{\xi_{\lambda} \cap ( I_0 \cup I_1) = \emptyset \}.
$$  

\begin{lemm}
\label{lemrope}
If $0 < \lambda < \lambda_D$, then
 $\Pr[E_{\lambda}] >0$.
\end{lemm}
%Let $D_k$ be the Euclidean diameter of 
%$W_{k,\lambda}$. 
\begin{proof}
Fix $\lambda \in (0,\lambda_D)$. Then by  Lemma \ref{lemGouere}.
$$
\sum_{k \geq 1} \Pr[ D(W_{k,\lambda}) > k/2] 
= \sum_{k \geq 1}
 \Pr[ D(W_{0,\lambda}) > k/2]  \leq \E[2D(W_{0,\lambda})] < \infty.
$$
Choose $k_0 \in \N$ with $k_0 >2$, such that 
$\sum_{k \geq k_0} \Pr[ D(W_{k,\lambda}) > k/2] < 1/2$.
 Then by the union bound and complementation, 
$\Pr[ \cap_{k=k_0}^\infty \{D(W_{k,\la}) \leq k/2\}] \geq 1/2$.
Moreover $\Pr[ \cap_{k=0}^{k_0} \{\xi_\la \cap I_k = \emptyset\}] >0.$
Hence by the  Harris-FKG inequality,
$$
\Pr [E_\la ] \geq \Pr \left[
\left( \cap_{k=k_0}^\infty \{D(W_{k,\la}) \leq k/2\} \right)
\cap
\left( \cap_{k=0}^{k_0} \{\xi_\la \cap I_k = \emptyset\} \right) 
\right] 
>0.
$$
\end{proof}

\begin{lemm}
\label{unilem}
Suppose that $A \subset \R^d$ is closed, connected and unbounded, and
 that
 $\R^d \setminus A$ has an unbounded connected component.
Then $\partial A$, the boundary of $A$, has an unbounded connected component.  
\end{lemm}
\begin{proof}
Let $B$ be an 
 unbounded component of $\R^d \setminus A$.
 Denote  the closure of 
$B$ by
 $\overline{B}$.
Then both $\overline{B}$ and $\R^d \setminus B$ are 
closed and connected.  
By the unicoherence of $\R^d$ \cite{Uniref}, the set
 %$\partial_{ext}(A) :=
$ \overline{B}  \cap A = 
\overline{B} \cap (\R^d \setminus B) $
 is connected.  Moreover it
% $ A \cap \overline{B} $ 
  is unbounded,  and contained in $\partial A$. 
\end{proof}

Given $\eps >0$, let $\txi_{\lambda,\eps} := 
%\cup_{y_k \in \Po_\lambda: \rho_k >0} B(y_k, \eps \rho_k)$ and let 
\cup_{k : \rho_k >0} B(y_k, \eps \rho_k)$ and let 
 $\txi^*_{\lambda,\eps} := 
\R^d \setminus 
 \txi_{\lambda,\eps} $. 
Let $\xi_\lambda^0 := \cup_{\{k: r_k =0\}} \{y_k\} $, the union of 
balls of radius zero  contributing to our Boolean model ($\txi_{\lambda,\eps}$
 is the
union of all the other balls, scaled by $\eps$).
If $\Pr[\rho =0]>0$ then $\xi_\lambda^0$ is 
(almost surely)
non-empty but locally finite.

\begin{lemm}
%Assume $\Pr[\rho =0]=0$.
Let $\eps \in (0,1)$.
If $E_\lambda$ occurs
% and also $\xi_\lambda$ does not percolate, 
then $\txi_{\lambda,\eps}^* \cup \xi^0_\lambda$ percolates.
\end{lemm}
\begin{proof}
Suppose $E_\lambda $ occurs. Let $A$ be the union of
 the half-line $[1,\infty) \times \{o'\}$, with all
components of $\xi_\lambda$ intersecting this half-line.
 Then $A$ is connected, unbounded,
 and contained in
the half-space $[1,\infty) \times \R^{d-1}$ so
$o$ lies in an unbounded component of $\R^d \setminus A$.
%
%which we denote by $A^c_o$.  Denote  the closure of $A^c_o$ by
% $\overline{A^c_o}$.  Then both $\overline{A^c_o}$ and
% $\R^d \setminus A^c_o$ are connected.
%
%Set $\partial_{ext}(A) := A \cap \overline{A^c_o} = 
%\overline{A^c_o} \cap (\R^d \setminus A^c_o) $.
%By the unicoherence of $\R^d$ \cite{Uniref}, this set is connected.
% Moreover it is unbounded.
%If we take open balls for our Boolean model, then $\partial_{ext}(A)$ is itself
%contained in $\xi_\lambda^*$ so we are done. 
Therefore by Lemma \ref{unilem}, $\partial A$ has an unbounded
 connected component.
% which we denote $\partial_{ext}A$.

No point of $\partial A$ lies  in 
the interior of any of the balls $B(y_k,\rho_k)$.
Therefore %each point of
 $\partial A \subset
\txi^*_{\lambda,\eps} \cup \xi_\lambda^0$.
Thus $ \txi^*_{\lambda,\eps} \cup \xi_\lambda^0$ has an
unbounded connected subset.
\end{proof}

\begin{proof}[Proof of (\ref{eq:D})]
Assume $\lambda_D >0$ (else there is nothing to prove).
 Suppose $\lambda \in (0, \lambda_D)$
and $\eps \in (0,1)$. 
By the last two lemmas,
with strictly positive probability
% provided $\Pr[\rho=0]=0$,
 %for $\lambda = \lambda_1$, 
the set $\txi_{\lambda,1-\eps}^* \cup \xi_\lambda^0$ percolates.
Almost surely,
 $\txi_{\lambda,1-\eps}^*  $ is open,
 $\xi_\lambda^0$ is locally finite and
all points of
 $\xi_\lambda^0$ lie either in $\txi_{\lambda,1-\eps}^* $ or
in the interior of 
  $\txi_{\lambda,1-\eps} $. Therefore if 
the set
 $\txi_{\lambda,1-\eps}^* \cup \xi_\lambda^0$ percolates, so
does $\txi_{\lambda,1-\eps}^*  $, and so does 
$\txi_{\lambda,1-\eps}^* \setminus \xi_\lambda^0$.
%points of $\xi_\lambda^0$ to the latter set does not affect
%its connectivity.
Thus the set
$
\txi_{\lambda,1-\eps}^* \setminus \xi_{\lambda}^0, 
$
which is equal to
$ \R^d \setminus \cup_k B(y_k,(1-\eps)\rho_k )$,
 percolates with strictly
positive probability,
and  hence by ergodicity, with probability 1. 
Hence by scaling (see \cite{MR}) the set $\xi_{(1-\eps)^{d}\lambda}^*$ also
percolates almost surely, so that
$\lambda_c^* \geq (1-\eps)^{d} \lambda$, and therefore 
$\lambda_c^* \geq \lambda_D$.
\end{proof}
{\bf Acknowledgement.} I thank the referees for some helpful
remarks.

\end{document}